\documentclass[a4paper,10pt]{article}
\usepackage{amsmath,amssymb,amsfonts,amsthm}
\usepackage{graphicx,color,subfigure}
\usepackage{enumerate,multirow}
\usepackage{fullpage}
\usepackage[]{hyperref}
\usepackage[normalem]{ulem}

\newcommand\RR{\ensuremath{\mathbb{R}}}

\newcommand\eps{\ensuremath{\varepsilon}}
\newcommand\Bn{\ensuremath{\mathbb{B}^n}}

\newtheorem{thm}{Theorem}[section]
\newtheorem{prop}[thm]{Proposition}

\newtheorem{lem}[thm]{Lemma}

\title{Pleijel's nodal domain theorem for Neumann and Robin eigenfunctions}
\author{Corentin L\'ena\footnote{Dipartimento di Matematica \emph{Giuseppe Peano}, Universit\`a degli Studi di Torino, Via Carlo Alberto, 10, 10123 Torino (TO), Italia, \texttt{clena@unito.it}}}
	
\begin{document}
	\maketitle
	
	\begin{abstract}
		In this paper, we show that equality in Courant's nodal domain theorem can only be reached for a finite number of eigenvalues of the Neumann Laplacian, in the case of an open, bounded and connected set in $\RR^n$ with a $C^{1,1}$ boundary. This result is analogous to Pleijel's nodal domain theorem for the Dirichlet Laplacian (1956). It confirms, in all dimensions, a conjecture formulated by Pleijel, which had already been solved by I.~Polterovich for a two-dimensional domain with a piecewise-analytic boundary (2009). We also show that the argument and the result extend to a class of Robin boundary conditions.
	\end{abstract}
	
	
	
\section{Introduction}
\label{secIntro}	
\subsection{Problem and results}
The main objective of this paper is to extend Pleijel's nodal domain theorem to eigenfunctions of the Laplacian which satisfy a Neumann boundary condition. Let  $\Omega\subset \RR^n$, with $n\ge2$, be a connected open set, which we assume to be bounded, with a sufficiently regular boundary. For technical reasons (appearing in Section \ref{secIPP}), we ask for $\partial \Omega$ to be $C^{1,1}$. In the rest of the paper, we denote by $-\Delta_{\Omega}^N$ the self-adjoint realization of the (non-negative) Laplacian in $\Omega$, with the Neumann boundary condition, and by $(\mu_k(\Omega))_{k\ge 1}$ its eigenvalues, arranged in non-decreasing order and counted with multiplicities. Similarly, in the case of the Dirichlet boundary condition, we denote the self-adjoint realization of the Laplacian by $-\Delta_{\Omega}^D$ and its eigenvalues by $(\lambda_k(\Omega))_{k\ge 1}$.

For any function $f$ continuous in $\Omega$, we call \emph{nodal set} of $f$ the closed set 
\begin{equation*}
 \mathcal{N}(f):=\overline{\{x \in \Omega \,;\, f(x)=0\}}
\end{equation*}
and \emph{nodal domains} the connected components of $\Omega \setminus \mathcal N(f)$. We denote by $\nu(f)$ the cardinal of the set of nodal domains. We are interested in estimating $\nu(u)$ from above when $u$ is an eigenfunction of $-\Delta^N_{\Omega}$. A fundamental result of this type was first obtained by R.~Courant in 1923 (see \cite{Cou23Nodal} or \cite[VI.6]{CouHil53}).
\begin{thm}
	\label{thmCourant}
	If $k$ is a positive integer and $u$ an eigenfunction associated with $\lambda_k(\Omega)$ or $\mu_k(\Omega)$, $\nu(u)\le k$.
\end{thm}
\AA.~Pleijel showed in 1956 that, in the Dirichlet case, equality in the previous theorem can only occur for a finite number of eigenvalues. He originally proved it for domains in $\RR^2$ \cite{Ple56}. The result was extended by J. Peetre in 1957  to some domains on two-dimensional Riemannian manifolds \cite{Peetre1957Courant}, and a general version, valid for $n$-dimensional Riemannian manifolds with or without boundary, was obtained by P.~B\'erard and D.~Meyer in 1982 \cite[II.7]{BerMey82}. In those three works, the authors actually proved a stronger result, in the form of an asymptotic upper bound. To state it, let us denote by $\nu_k^D(\Omega)$ the largest possible value of $\nu(u)$, when $u$ is an eigenvalue of $-\Delta_{\Omega}^D$ associated with $\lambda_k(\Omega)$. Let us also define
\begin{equation*}
\gamma(n):=\frac{2^{n-2}n^2\Gamma\left(\frac{n}2\right)^2}{j_{\frac{n}2-1,1}^n},
\end{equation*}
where $j_{\frac{n}2-1,1}$ is the smallest positive zero of the Bessel function of the first kind $J_{\frac{n}2-1}$. We recall the inequality $\gamma(n)<1$, proved in \cite[II.9]{BerMey82} (see also \cite[Section 5]{HelPer16PAMS} for more precise results).
\begin{thm} \label{thmPleijelDir} If $\Omega\subset \RR^n$ is an open, bounded, and connected set which is Jordan measurable,
	\begin{equation*}
		\limsup_{k\to +\infty}\frac{\nu_k^D(\Omega)}{k}\le \gamma(n).
	\end{equation*} 	
\end{thm}
Theorem \ref{thmPleijelDir} is actually proved in \cite{BerMey82} for closed Riemannian manifolds, or Riemannian manifolds with smooth boundary in the Dirichlet case. However, the results in \cite{BerMey82} do not include the Neumann case. Let us note that the Jordan measurability is imposed in Theorem \ref{thmPleijelDir} so that the Weyl's asymptotics holds for the sequence $(\lambda_k(\Omega))_{k\ge 1}$ (see \cite[Section XIII.15]{ReeSim78}). The hypothesis that $\partial \Omega$ is $C^{1,1}$ would be more than sufficient.
 
The constant $\gamma(n)$ has the following interpretation:
\begin{equation*}
	\gamma(n)=\frac{(2\pi)^n}{\lambda_1(\Bn)^{\frac{n}2}\omega_n},
\end{equation*}
where $\omega_n$ is the volume of the unit ball in $\RR^n$, and $\Bn$ is a ball of volume $1$ in $\RR^n$. It was proved in \cite[Section 5]{HelPer16PAMS} that the sequence $(\gamma(n))_{n\ge 2}$ is strictly decreasing and converges to $0$ exponentially fast.

Since $\gamma(n)<1$, Theorem \ref{thmPleijelDir} implies that there exists a finite smallest rank $k_D(\Omega)\ge 1$ such that, for all $k>k_D(\Omega)$, an eigenfunction  of $-\Delta_{\Omega}^D$ associated with $\lambda_k(\Omega)$ as strictly less than $k$ nodal domains. Two recent papers \cite{BerHel2016Geom,BergGittins2016JST} give upper bounds of $k_D(\Omega)$ in term of the geometry of $\Omega$.

As in the Dirichlet case, let us denote by $\nu_k^N(\Omega)$ the largest possible value for $\nu(u)$ when $u$ is an eigenfunction of $-\Delta_{\Omega}^N$ associated with $\mu_k(\Omega)$. The question of finding an asymptotic upper bound of $\nu_k^N(\Omega)$ was already raised by \AA. Pleijel, who showed that, for a square, the same upper bound as in the Dirichlet case holds true \cite[Section 7]{Ple56}. To our knowledge, the most general result known so far involving $\nu_k^N(\Omega)$ has been obtained by I.~Polterovich in 2009 \cite{Pol09}. His proof uses estimates by J.~A.~Toth and S.~Zelditch \cite{TotZel09} of the number of nodal lines touching the boundary, and is therefore restricted to two-dimensional domains with quite regular boundaries.
\begin{thm} \label{thmPolterovich} If $\Omega\subset \RR^2$ is an open, bounded, and connected set with a piecewise-analytic boundary,
	\begin{equation*}
	\limsup_{k\to +\infty}\frac{\nu_k^N(\Omega)}{k}\le \gamma(2)=\frac4{j_{0,1}^2}.
	\end{equation*} 	
\end{thm} 
In this paper, we prove the following result, valid in any dimension.
\begin{thm}
	\label{thmPleijelNeu}
	If $\Omega \subset \RR^n$, with $n\ge 2$, is an open, bounded, and connected set with a $C^{1,1}$ boundary,
	\begin{equation*}
	\limsup_{k\to +\infty}\frac{\nu_k^N(\Omega)}{k}\le \gamma(n).
	\end{equation*}
\end{thm}

Furthermore, our proof can be quite easily extended to some Robin-type boundary conditions. Let us be more specific: we assume that $\Omega$ satisfies the same hypotheses as above and that $h$ is a Lipschitz function in $\overline \Omega$ such that $h\ge 0$ on $\partial \Omega$. By analogy with the Dirichlet and Neumann cases, we denote by $-\Delta_{\Omega}^{R,h}$ the self-adjoint realization of the Laplacian in $\Omega$ with the Robin boundary  condition
\begin{equation*}
\frac{\partial u}{\partial \nu}+hu=0 \mbox{ on }\partial \Omega.
\end{equation*}
Here $\frac{\partial u}{\partial \nu}$ is the exterior normal derivative. We denote by $(\mu_k(\Omega,h))_{k\ge 1}$ the eigenvalues of $-\Delta_{\Omega}^{R,h}$ and by $\nu_k^{R}(\Omega,h)$ the maximal number of nodal domains for an eigenfunction of $-\Delta^{R,h}_{\Omega}$ associated with $\mu_k(\Omega,h)$. Let us note that in dimension $2$, this eigenvalue problem gives the natural frequencies of a membrane elastically held at its boundary \cite[9.5]{Weinstock1974CV}. The condition $h\ge 0$ implies that each point $x$ on the boundary is subject either to no force (if $h(x)=0$) or to a binding elastic force pulling it back to its equilibrium position (if $h(x)>0$).  We prove the following result.
\begin{thm}
	\label{thmPleijelRob}
	If $\Omega \subset \RR^n$, with $n\ge 2$, is an open, bounded, and connected set with a $C^{1,1}$ boundary, and if $h$ is a Lipschitz function in $\overline \Omega$ with $h\ge 0$ on $\partial \Omega$,
	\begin{equation*}
	\limsup_{k\to +\infty}\frac{\nu_k^R(\Omega,h)}{k}\le \gamma(n).
	\end{equation*}
\end{thm}
Theorem \ref{thmPleijelNeu} is of course a special case of Theorem \ref{thmPleijelRob}, corresponding to $h=0$. However, in order to make the argument more readable, we prefer to treat first the Neumann case, and then outline the few changes to be made in order to prove Theorem  \ref{thmPleijelRob}.

Let us note that the same constant $\gamma(n)$ appears in the Dirichlet, Neumann, and Robin cases. It is known from the work of J. Bourgain \cite{Bourgain15Pleijel} and S. Steinerberger \cite{Steinerberger2014Pleijel}  that this constant is not optimal when $n=2$ in the Dirichlet case. See \cite{HelHof15Review} for an extensive discussion, in connection with minimal partition problems.  I.~Polterovich \cite{Pol09} conjectures that for a sufficiently regular open set $\Omega \subset \RR^2$, 
\begin{equation*}
	\limsup_{k\to +\infty}\frac{\nu_k^D(\Omega)}{k}\le \frac2{\pi}
\end{equation*}
and
\begin{equation*}
\limsup_{k\to +\infty}\frac{\nu_k^N(\Omega)}{k}\le \frac2{\pi}.
\end{equation*}
The constant $\frac2{\pi}$ is the smallest possible, as can be seen by considering rectangles \cite{Pol09,HelHof15Review,BonHel2015Chapter}.  Let us finally point out that  the analogue to Theorems \ref{thmPleijelDir}, \ref{thmPleijelNeu}, and \ref{thmPleijelRob}, with the same constant $\gamma(n)$, holds for the Schr\"odinger operator $-\Delta+V$ in $\RR^n$, for a class of radially symmetric potentials $V$. This was shown recently by P.~Charron \cite{Charron2015HO} for the harmonic oscillator and P.~Charron, B.~Helffer, and T.~Hoffmann-Ostenhof in a more general situation \cite{ChaHefHof2016}.

\subsection{Overview of the paper}
Let us now introduce the main ideas of the paper. The proof of Theorem \ref{thmPleijelNeu} is given in Section \ref{secProof}, and follows quite closely Pleijel's original argument \cite{Ple56}. This consisted in obtaining a control of $\nu_k^D(\Omega)$ in terms of $\lambda_k(\Omega)$ and $|\Omega|$, the volume of $\Omega$, by applying the Faber-Krahn inequality to each nodal domain of an eigenfunction $u$, associated with $\lambda_k(\Omega)$. The key fact at this point is the equality $\lambda_k(\Omega)=\lambda_1(D)$ where $D$ is a nodal domain. The upper bound in Theorem \ref{thmPleijelDir} then follows from Weyl's law. In the Neumann case, the same method cannot be applied to the nodal domains touching the boundary of $\Omega$, since the eigenfunctions do not satisfy a Dirichlet boundary condition there. The proof in \cite{Pol09} relied on the fact that the number of nodal domains touching the boundary is controlled by $\sqrt{\mu_k(\Omega)}$, under the hypotheses of Theorem \ref{thmPolterovich}. As far as we know, a similar estimate is not available in dimension higher than $2$, nor for a less regular boundary. To overcome this obstacle, we classify the nodal domains of the eigenfunction $u$ into two types: those for which the $L^2$-norm of $u$ is mostly concentrated inside $\Omega$ (\emph{bulk domains}), and those for which a significant proportion of the $L^2$-norm is concentrated near $\partial \Omega$ (\emph{boundary domains}). To control the number of boundary domains, we reflect them through $\partial \Omega$ before applying the Faber-Krahn inequality. To make this approach precise, we use some rather standard partition-of-unity arguments. Let us point out that P.~B\'erard and B.~Helffer propose a closely related strategy  for proving a version of Pleijel's theorem, on a manifold  with boundary $M$, in the Neumann case.  They suggest to consider the double manifold $\widehat M$, obtained by gluing two copies of $M$ along the boundary $\partial M$. They obtain  in this way a manifold which is symmetric with respect to $\partial M$. They can then identify the Neumann eigenfunctions on $M$ with the symmetric eigenfunctions of the Laplacian on $\widehat M$. The results in \cite{BerMey82} would then give
\[\lim_{k\to +\infty}\frac{\nu_k^N(M)}k\le 2\gamma(n),\]
with $2\gamma(n)<1$ for $n \ge 3$ (see \cite[Section 5]{HelPer16PAMS}). However, for this geometric approach to  work, one has to require that the metric on $\widehat M$ is sufficiently regular, so that the asymptotic isoperimetric inequality of \cite[II.15]{BerMey82} holds, and therefore also the asymptotic Faber-Krahn inequality of \cite[II.16]{BerMey82}. This would impose strong constraints on $\partial M$, for instance that $\partial M$ is totally geodesic, which are unnecessary in the approach of the present paper. 

For the sake of completeness, we give in Section \ref{secIPP} the proof of two technical results, which are used in Section \ref{secProof}. The first is a regularity result up to the boundary for Neumann eigenfunctions.
\begin{prop} \label{propReg} Let $\Omega \subset \RR^n$ be  an open, bounded, and connected set with a $C^{1,1}$ boundary. An eigenfunction $u$ of $-\Delta_{\Omega}^N$ belongs  to $C^{1,1^-}\left(\overline \Omega\right):=\bigcup_{\alpha \in (0,1)}C^{1,\alpha}\left(\overline \Omega\right)$. In particular, $u\in C^1\left(\overline \Omega\right)$. 
\end{prop}
The proof follows a remark from \cite[1.2.4]{Hen06} and uses the regularity results for elliptic boundary value problems contained in the classical monograph \cite{Grisvard1985}.

The second result, used repeatedly in Section \ref{secProof}, is the following Green's formula for Neumann eigenfunctions. 
\begin{prop}
	\label{propGreen}  Let $\Omega \subset \RR^n$ be  an open, bounded, and connected set with a $C^{1,1}$ boundary. If $u$ is an eigenfunction of $-\Delta_{\Omega}^N$ associated with the eigenvalue $\mu$, and if $D$ is a nodal domain of $u$, then
	\begin{equation}
	\label{eqGreen}
	\int_{D}\left|\nabla u\right|^2\,dx=\mu\int_{D}u^2\,dx.
	\end{equation}
\end{prop}
In dimension $2$, the nodal set of $u$ is the union of a finite number of $C^1$ curves, possibly crossing or hitting the boundary of $\Omega$ at a finite number of points, where they form equal angles (see for instance \cite[Section 2]{HelHofTer09}). In particular, this implies that $D$ is a Lipschitz domain (even piecewise-$C^1$). We can therefore apply Green's formula for bounded Lipschitz domains (see for instance \cite[Theorem 1.5.3.1]{Grisvard1985}),   
and obtain Equation \eqref{eqGreen} directly. In higher dimension, there exists as far as we know no proof that the nodal domains are Lipschitz (see however \cite{HardtEtAl1999Critical}). A way around this difficulty is indicated in \cite[Appendix D]{BerMey82}. The authors approximate a nodal domain by super-level sets of the eigenfunction for regular values. The boundary of these sets is regular enough to apply Green's formula, and Sard's theorem provides a sequence of regular values converging to $0$. We give a proof a Proposition \ref{propGreen} along similar lines. This method is also used in Section \ref{subSecRef} to carry out the reflection argument.

In Section \ref{secRobin}, we indicate the changes to be made in order to treat a Robin boundary condition. We give a precise formulation of the eigenvalue problem and prove Theorem \ref{thmPleijelRob}.

	\section{Proof of the main theorem}
	\label{secProof}
	
	\subsection{Preliminaries}
	
	For any $\delta>0$, we write
	\begin{equation*}
	\partial\Omega_{\delta}:=\{x\in \RR^n\,;\, \mbox{dist}(x,\partial\Omega)<\delta\}.
	\end{equation*}
	and
	\begin{equation*}
	\partial\Omega_{\delta}^+:=\{x\in \Omega\,;\, \mbox{dist}(x,\partial\Omega)<\delta\}.
	\end{equation*}
	
	Let us first note that, since the boundary of $\Omega$ is $C^{1,1}$, we can locally straighten it. More explicitly, there exists a finite covering $(U_i)_{1\le i\le N}$ of $\partial \Omega$ by coordinate charts. We mean by coordinate charts that, for $i\in\{1,\dots,N\}$, $U_i$ is an open set in $\RR^n$, and there exists an open ball $B(0,r_i)$ in $\RR^n$ and a $C^{1,1}$ diffeomorphism $\psi_i: B(0,r_i) \to U_i$, such that $\psi_i$ and $\psi_i^{-1}$ are bounded, with first order derivatives bounded and Lipschitz, and such that
	\begin{equation*}
	U_i\cap \Omega =\psi_i\left(B^+(0,r_i)\right)
	\end{equation*}  
	where
	\begin{equation*}
	B^+(0,r_i)=\{y=(y',y_n)\in \RR^{n-1}\times \RR,\,;\, y\in B(0,r_i)\mbox{ and }y_n>0\}.
	\end{equation*}
	There exists an associated family $(\chi_i)_{1\le i \le N}$ of $C^{1,1}$ non-negative functions and a positive constant $\delta_1$ such that
	\begin{enumerate}[i.]
		\item $\mbox{supp}(\chi_i)$ is compactly included in $U_i$ for $i\in \{1,\dots,N\}$;
		\item $\sum_{i=1}^N\chi_i^2\le 1$ in $\RR^n$ and $\sum_{i=1}^N\chi_i^2\equiv1$ in $\partial \Omega_{\delta_1}$.
	\end{enumerate}
	Let us note that $N$, $\delta_1$, and the family $(\chi_i)_{1\le i \le N}$ depend only on $\Omega$, and are fixed in the rest of this section.
	As a consequence of this local straightening of the boundary, we have the existence of partitions of unity adapted to our problem.
		\begin{lem}
			\label{lemPartition}
			There exist two positive constants $0<a<A$ such that, for all $0<\delta<\delta_1/A$, there exists two non-negative functions $\varphi_0^{\delta}$ and $\varphi_1^{\delta}$, $C^{1,1}$ in $\Omega$, satisfying
			\begin{enumerate}[i.]
				\item $(\varphi_0^{\delta})^2+(\varphi_1^{\delta})^2\equiv 1$ on $\Omega$;
				\item $\mbox{supp}(\varphi_0^{\delta})\subset \Omega\setminus \overline{ \partial\Omega_{a\delta}^+}$ and $\mbox{supp}(\varphi_1^{\delta})\subset \partial \Omega_{A\delta}^+$;
				\item $\left|\nabla \varphi_i^{\delta}(x)\right|\le \frac{C}{\delta}$ for $x\in \Omega$ and $i\in\{0,1\}$, with $C$ a constant independent of $\delta$.
			\end{enumerate}
		\end{lem}
		
		\begin{proof}
			This construction is rather standard, and we merely give an outline of the demonstration. It is enough to find a $C^{1,1}$ function $f^{\delta}$ such that 
			\begin{enumerate}[a.]
				\item $0\le f^{\delta}\le 1$ on $\Omega$;
				\item $f^{\delta}\equiv 1$ on $\overline{\partial\Omega_{a\delta}^+}$ and $\mbox{supp}(f^{\delta})\subset \partial \Omega_{A\delta}^+$;				
				\item $\left|\nabla f^{\delta}(x)\right|\le \frac{K}{\delta}$ for $x\in \Omega$ and $i\in\{0,1\}$, with $K$ a constant independent of $\delta$.
			\end{enumerate}
				Indeed, we can then set
				\begin{equation*}
					\varphi_1^{\delta}(x)=\frac{f^{\delta}(x)}{\sqrt{f^{\delta}(x)^2+(1-f^{\delta}(x))^2}} \mbox{ and } \varphi_0^{\delta}(x)=\frac{1-f^{\delta}(x)}{\sqrt{f^{\delta}(x)^2+(1-f^{\delta}(x))^2}},
				\end{equation*}
				and we obtain functions satisfying Properties $i-iii$ of Lemma \ref{lemPartition}. In order to construct $f^{\delta}$, we fix $0<b<B$ such that for all $i\in\{1,\dots,N\}$ and all $y=(y',y_n)\in B^+(0,r_i)$, 
				\begin{equation*}
						by_n\le\mbox{dist}(\psi_i(y),\partial\Omega)\le By_n.
				\end{equation*}   
				We also fix a non-increasing smooth function $g:\RR\to\RR$ such that $g(t)=1$ for all $t\in (-\infty,1/4]$ and $g(t)=0$ for all $t\in[3/4,+\infty)$. We now define the function $f_i^{\delta}$ for each $i\in\{1,\dots,N\}$ by $f_i^{\delta}(x)=0$ if $x \notin U_i$ and
				\begin{equation*}
					f_i^{\delta}(x)=g\left(\frac{y_n}{\delta}\right)\chi_i(x)
				\end{equation*}
				for $x\in U_i$, with $y=(y',y_n)=\psi_i^{-1}(x)$. Then the function $f^{\delta}=\sum_{i=1}^N \left(f_i^{\delta}\right)^2$ satisfies Properties a-c with $a=b/4$ and $A=B$.			
		\end{proof}
		
	Let us now write $\mu=\mu_k(\Omega)$, with $k\ge 2$, and let $u$ be an associated eigenfunction with $\nu_k^N(\Omega)$ nodal domains. We set $\delta:=\mu^{-\theta}$, with $\theta$ a positive constant to be determined later. We write $u_0:=\varphi^{\delta}_0u$ and $u_1:=\varphi^{\delta}_1u$. According to property i of Lemma \ref{lemPartition}, we have, for every nodal domain $D$,
	\begin{equation*}
		\int_{D}u^2\,dx=\int_{D}u_0^2\,dx+\int_{D}u_1^2\,dx.
	\end{equation*}
	
	\subsection{First main step: two types of nodal domains}
	We fix $\eps>0$, and we distinguish between the \emph{bulk domains}, i.e. the domains $D$ satisfying
	\begin{equation*}
		\int_{D}u_0^2\,dx\ge (1-\eps)\int_{D}u^2\,dx,
	\end{equation*}
	and the \emph{boundary domains}, i.e. the domains $D$ satisfying
	\begin{equation*}
	\int_{D}u_1^2\,dx> \eps\int_{D}u^2\,dx,
	\end{equation*}
	 We denote by $\nu_0(\eps,\mu)$ the number of bulk domains, and the bulk domains themselves by $D_1^0,\dots,D^0_{\nu_0(\eps,\mu)}$. Similarly, the number of boundary domains and the boundary domains themselves are denoted by $\nu_1(\eps,\mu)$ and $D^1_1,\dots,D^1_{\nu_1(\eps,\mu)}$ respectively.
	 
	 \subsection{Bulk domains}
	 
	 We begin by giving an upper bound of the number of bulk domains $\nu_0(\eps,\mu)$. Let us fix $j\in\{1,\dots, \nu_0(\eps,\mu)\}$. According to the Faber-Krahn inequality (see for instance \cite[3.2]{Hen06}) and the variational characterization of $\lambda_1(D^0_j)$, we have 
	 \begin{equation}
		 \label{eqIneqFKBulk}
		 \lambda_1(\Bn)\left|D^0_j\right|^{-\frac2n}\le \frac{\int_{D^0_j}\left|\nabla u_0\right|^2\,dx}{\int_{D^0_j}u_0^2\,dx}\le \frac1{1-\eps}\frac{\int_{D^0_j}\left|\nabla u_0\right|^2\,dx}{\int_{D^0_j}u^2\,dx}.
	 \end{equation}
	 We have, in $\Omega$,
	 \begin{equation*}
		 \nabla u_0=\varphi^{\delta}_0\nabla u+ u \nabla \varphi_0^{\delta},
	 \end{equation*}
	 and therefore, according to Young's inequality,
	 \begin{equation*}
		 \left|\nabla u_0\right|^2\le (1+\eps)(\varphi^{\delta}_0)^2\left|\nabla u\right|^2+\left(1+\frac1{\eps}\right) \left|\nabla \varphi_0^{\delta}\right|^2 u^2.
	 \end{equation*}
	 Integrating over $D_j^0$, and using property i and iii of Lemma \ref{lemPartition}, we find
	 \begin{equation}
		 \label{eqIneqEnerBulk}
		 \int_{D^0_j}\left|\nabla u_0\right|^2\,dx\le (1+\eps)\int_{D^0_j}\left|\nabla u\right|^2\,dx+\left(1+\frac1{\eps}\right) \frac{C^2}{\delta^2}\int_{D^0_j} u^2\,dx.
	 \end{equation}
	 Injecting Inequality \eqref{eqIneqEnerBulk} into Inequality \eqref{eqIneqFKBulk}, we find
	 \begin{equation*}
		 \lambda_1(\Bn)\left|D^0_j\right|^{-\frac2n}\le \frac{1+\eps}{1-\eps}\frac{\int_{D^0_j}\left|\nabla u\right|^2\,dx}{\int_{D^0_j}u^2\,dx}+\frac{\left(1+\frac1{\eps}\right)C^2}{(1-\eps)\delta^2}
	 \end{equation*}
According to Proposition \ref{propGreen}, we have 
\begin{equation*}
	\int_{D^0_j}\left|\nabla u\right|^2\,dx=\mu \int_{D^0_j} u^2\,dx,
\end{equation*}
and therefore
	 \begin{equation*}
		 \lambda_1(\Bn)\left|D^0_j\right|^{-\frac2n}\le \frac{1+\eps}{1-\eps}\mu+\frac{1+\frac1{\eps}}{1-\eps}C^2\mu^{2\theta}.
	 \end{equation*}
	 We therefore obtain
	 \begin{equation*}
	 1\le\frac{\left|D^0_j\right|}{\lambda_1(\Bn)^{\frac{n}2}}\left(\frac{1+\eps}{1-\eps}\mu+\frac{1+\frac1{\eps}}{1-\eps}C^2\mu^{2\theta}\right)^{\frac{n}2}
	 \end{equation*}
	 and, summing over $j\in\{1,\dots, \nu_0(\eps,\mu)\}$, we get 
	 \begin{multline}
	 \label{eqIneqN0}
		 \nu_0(\eps,\mu)\le\frac1{\lambda_1(\Bn)^{\frac{n}2}}\left(\frac{1+\eps}{1-\eps}\mu+\frac{1+\frac1{\eps}}{1-\eps}C^2\mu^{2\theta}\right)^{\frac{n}2}\left|\bigcup_{j=1}^{\nu_0(\eps,\mu)}D^0_j\right|\\\le \frac1{\lambda_1(\Bn)^{\frac{n}2}}\left(\frac{1+\eps}{1-\eps}\mu+\frac{1+\frac1{\eps}}{1-\eps}C^2\mu^{2\theta}\right)^{\frac{n}2}\left|\Omega\right|.
	 \end{multline}
	 
	 \subsection{Boundary domains}
	 
	 Let us now give an upper bound of the number of boundary domains $\nu_1(\eps,\mu)$. We further decompose the function $u_1$ in the following way. Using the family $(\chi_i)_{1\le i \le N}$ introduced at the beginning of this section, we set $u_1^i:=\chi_i u_1$. Let us now fix $j\in \{1,\dots,\nu_1(\eps,\mu)\}$. According to Property ii for the family $(\chi_i)_{1\le i \le N}$ (here we use the assumption $\delta<\delta_1/A$), we have 
	 \begin{equation*}
		 \int_{D^1_j}u_1^2\,dx=\sum_{i=1}^N \int_{D^1_j}(u_1^i)^2\,dx.
	 \end{equation*}
	 As a consequence, there exists $i_j\in \{1,\dots,N\}$ such that 
	 \begin{equation}
	 \label{eqIneqBoundaryMass}
		 \int_{D^1_j}(u_1^{i_j})^2\,dx_\ge \frac1N \int_{D^1_j}u_1^2\,dx > \frac{\eps}{N}\int_{D^1_j}u^2\,dx.
	 \end{equation}
	 Let us note that $u^{i_j}_1=\widetilde{\chi}u$, with $\widetilde{\chi}=\chi_{i_j}\varphi_j^{\delta}$. The function $\widetilde \chi$ is $C^{1,1}$, $0\le \widetilde \chi \le 1$, $\mbox{supp}(\widetilde \chi)\subset \partial \Omega_{A\delta}$, and 
	 \begin{equation}
	 \label{eqIneqGradBound}
		 \left|\nabla \widetilde \chi \right| \le \frac{C'}{\delta},
	 \end{equation}
	 with $C'$ a constant depending only on $\Omega$. 
	 
	 Up to replacing $u$ by $-u$, we assume that $u_1^{i_j}$ is non-negative in $D^1_j$.	We define the open set 
	\begin{equation*}
		U:=\{x\in D^1_j\,;\, u_1^{i_j}(x)\neq 0\}	\subset \partial \Omega_{A\delta}^+\cap U_{i_j}.
	\end{equation*}
	We now straighten the boundary locally, that is to say we set $v:=u^{i_j}_1\circ \psi_{i_j}$ and $V:=\psi_{i_j}^{-1}(U)$. The set $V$ is open and contained in $B^+(0,r_{i_j})$. There exists positive constants $C''$ and $C'''$, depending only on $\Omega$, such that 
	\begin{equation}
	\label{eqIneqArea}
		|V|\le C''|U|
	\end{equation}
	and 
	\begin{equation}
		\label{eqIneqQuotient}
		\frac{\int_{U}\left|\nabla u_1^{i_j}\right|^2\,dx}{\int_{U}(u_1^{i_j})^2\,dx}
		\le C'''\frac{\int_{V}\left|\nabla v\right|^2\,dy}{\int_{V}v^2\,dy}.
	\end{equation}
	
     \subsection{Second main step: reflection of the boundary domains}	
		\label{subSecRef}
	
	 The basic idea consists in extending $V$ and $v$ by reflection through the hyperplane $\{y_n=0\}$. We denote by $\sigma$ the reflection
	\begin{equation}
	\label{eqRefSet}
		\begin{array}{cccc}
			\sigma:&\RR^n=\RR^{n-1}\times\RR&\to&\RR^n=\RR^{n-1}\times\RR\\
					& y=(y',y_n)&\mapsto&(y',-y_n).
		\end{array}
	\end{equation}
	Intuitively, we define $V^R:=\mbox{Int}\left(\overline{V\cup\sigma(V)}\right)$, and the function $v^R$ in $V^R$ by 
	\begin{equation}
	\label{eqRefFunct}
		v^R(y)=\left\{\begin{array}{lcl}
							v(y)&\mbox{ if }& y\in V;\\
							v(\sigma(y))&\mbox{ if }& y \in \sigma(V).
					\end{array}
		\right.
	\end{equation}
	We expect $\left|V^R\right|=2\left|V\right|$, and $v^R\in H^1_0\left(V^R\right)$. Indeed, if $\partial V$ is regular enough, $v$ satisfies a Dirichlet boundary condition on $\partial V\setminus \{y_n=0\}$, and therefore  $v^R$ satisfies a Dirichlet boundary condition on $\partial V^R$. We would then apply the Faber-Krahn inequality to the domain $V^R$ to obtain a lower bound of the Rayleigh quotient
	\begin{equation*}
		\frac{\int_{V}\left|\nabla v\right|^2\,dy}{\int_{V}v^2\,dy}.
	\end{equation*}
	
	The above reasoning is of course not valid in general, since we do not know if $\partial V$ is regular enough.  To overcome this difficulty, we follow a method used in \cite[Appendix D]{BerMey82} (we use the same method to prove Proposition \ref{propGreen}). Let us first note that Proposition \ref{propReg} implies that $v$ can be extended to a function $w\in C^1\left(\overline{D}\right)$. For $\alpha>0$ small enough, we consider the super-level sets
	\begin{equation*} 		
		V_{\alpha}:=\{y\in V\,;\, w(y)>\alpha\}
	\end{equation*}
	and
	\begin{equation*} 		
		W_{\alpha}:=\{y\in \overline V\,;\, w(y)>\alpha\},
	\end{equation*}
	and the functions
	\begin{equation*}
		v_{\alpha}:=(v-\alpha)_+=\max(v-\alpha,0)
	\end{equation*}
	and
	\begin{equation*}
		w_{\alpha}:=(w-\alpha)_+.
	\end{equation*}
	If $\alpha$ is a regular value for the function $v$, the set $\Sigma_{\alpha}:=\partial V_{\alpha}\setminus \{y_n=0\}$ is a $C^1$ submanifold with  boundary in $\RR^n$, the normal being given by  $\nabla v(y)$ for any $y\in \Sigma_{\alpha}$. The boundary of $\Sigma_{\alpha}$ is 
	\begin{equation*}
		\gamma_{\alpha}:=\left\{y=(y',y_n)\in\overline{V}\,;\, y_n=0\mbox{ and } w(y)=\alpha\right\}.
	\end{equation*}
	Let us impose the additional condition that $\alpha$ is a regular value for the function
		\begin{equation*}
		\begin{array}{cccc}
		w_{|\Gamma}:&\Gamma&\to&\RR\\
		&y'&\mapsto&w(y',0),
		\end{array}
		\end{equation*}
	where $\Gamma$ is the open set in $\RR^{n-1}$ defined by
	\begin{equation*}
		\Gamma=\{y'\in \RR^{n-1}\,;\,(y',0)\in \overline{V}\mbox{ and } w(y',0)>0\}.
	\end{equation*} 
	Then, for any $y\in \gamma_{\alpha}$, the component of $\nabla w(y)$ tangential to $\{y_n=0\}$ is non-zero. This implies that that $\Sigma_{\alpha}$ touches the hyperplane $\{y_n=0\}$ transversally.
	Let us now denote by $V_{\alpha}^R$ and $v_{\alpha}^R$ the reflection of  $V_{\alpha}$ and $v_{\alpha}$ through $\{y_n=0\}$, defined in the same way as $V^R$ and $v^R$ in Equations \eqref{eqRefSet} and \eqref{eqRefFunct}. By our choice of $\alpha$, $\left|V^R_{\alpha}\right|=2\left|V_{\alpha}\right|$ and the function $v_{\alpha}^R$ belongs to $H^1_0\left(V^R_{\alpha}\right)$. 
	The Faber-Krahn inequality, applied to the open set $V^R_{\alpha}$, gives us 
	\begin{equation}
	\label{eqIneqRef}
		\lambda_1(\Bn)2^{-\frac2n}\left|V_{\alpha}\right|^{-\frac2n}=\lambda_1(\Bn)\left|V^R_{\alpha}\right|^{-\frac2n}\le \frac{\int_{V^R_{\alpha}}\left|\nabla v^R_{\alpha}\right|^2\,dy}{\int_{V^R_{\alpha}}(v^R_{\alpha})^2\,dy}=\frac{\int_{V_{\alpha}}\left|\nabla v_{\alpha}\right|^2\,dy}{\int_{V_{\alpha}}v_{\alpha}^2\,dy}.
	\end{equation}
	According to Sard's theorem, applied to the functions $v:D\to \RR$ and $w_{|\Gamma}:\Gamma\to\RR$, we can find a sequence $(\alpha_m)_{m\ge1}$ of positive  regular values for both functions satisfying $\alpha_m\to 0$. Using Inequality \eqref{eqIneqRef} for $\alpha=\alpha_m$ and passing to the limit, we find
	\begin{equation*}
		\lambda_1(\Bn)2^{-\frac2n}\left|V \right|^{-\frac2n}\le \frac{\int_{V}\left|\nabla v\right|^2\,dy}{\int_{V}v^2\,dy}.
	\end{equation*}
	Using Inequalities \eqref{eqIneqArea}, \eqref{eqIneqQuotient}, and \eqref{eqIneqBoundaryMass}, we get
	\begin{multline}
		\label{eqIneqFKBoundary}
		\lambda_1(\Bn)\left|D^1_j\cap \partial \Omega_{A\delta}^+\right|^{-\frac2n}\le\lambda_1(\Bn)\left|U\right|^{-\frac2n} \le \left(2C''\right)^{\frac2n}C'''\frac{\int_{U}\left|\nabla u_1^{i_j}\right|^2\,dx}{\int_{U}(u_1^{i_j})^2\,dx} \\ \le  \frac{N}{\eps}\left(2C''\right)^{\frac2n}C'''\frac{\int_{U}\left|\nabla u_1^{i_j}\right|^2\,dx}{\int_{D^1_j}u^2\,dx}.
	\end{multline}
	A computation similar to the one done for $u_0$, combined with Inequality \eqref{eqIneqGradBound}, gives us
	\begin{equation}
	\label{eqIneqEnerBoundary}
	\int_{U}\left|\nabla u_1^{i_j}\right|^2\,dx\le 2\int_{D^1_j}\left|\nabla u\right|^2\,dx+2\frac{(C')^2}{\delta^2}\int_{D^1_j} u^2\,dx.
	\end{equation}
	Combining Inequalities \eqref{eqIneqFKBoundary} and \eqref{eqIneqEnerBoundary}, as in the case of $u_0$, we obtain
	\begin{equation*}
		1\le \widetilde{C}|D_j^1\cap\partial \Omega_{A\delta}^+|\eps^{-\frac{n}2}\left(\mu+(C')^2 \mu^{2\theta}\right)^{\frac{n}2},
	\end{equation*}
	with $\widetilde C$ a constant depending only on $\Omega$. Summing over $j\in\{1,\dots,\nu_1(\eps,\mu)\}$, we get 
	\begin{equation*}
	\nu_1(\eps,\mu)\le \widetilde{C}|\partial \Omega_{A\delta}^+|\eps^{-\frac{n}2}\left(\mu+(C')^2 \mu^{2\theta}\right)^{\frac{n}2}.
	\end{equation*}
	Since $|\partial \Omega_{\delta}^+|\sim\mathcal{H}^{n-1}\left(\partial \Omega\right)A\delta$ as $\delta\to 0$, we obtain finally
	\begin{equation}
	\label{eqIneqN1}
	\nu_1(\eps,\mu)\le \widetilde{C}'\mu^{-\theta}\eps^{-\frac{n}2}\left(\mu+(C')^2\mu^{2\theta}\right)^{\frac{n}2},
	\end{equation}
	with $\widetilde C'$ a constant depending only on $\Omega$.
	
	\subsection{End of the proof}
	
	We now fix $\theta \in (0,1/2)$, for instance $\theta=1/4$, and consider the limits when  $k\to +\infty$, keeping $\eps$ fixed (we recall that $\mu=\mu_k(\Omega)$).
	We have 
	\begin{equation*}
	 \limsup_{k\to +\infty}\frac{\nu_k^N(\Omega)}{k} \le \limsup_{k\to+\infty}\frac{\nu_0(\eps,\mu_k(\Omega))}{k}+\limsup_{k\to+\infty}\frac{\nu_1(\eps,\mu_k(\Omega))}{k}.
	\end{equation*}
	Using Inequality \eqref{eqIneqN0}, we get 
	\begin{equation*}
		\limsup_{k\to+\infty}\frac{\nu_0(\eps,\mu_k(\Omega))}{k}\le \limsup_{k\to+\infty} \frac1{k\lambda_1(\Bn)^{\frac{n}2}}\left(\frac{1+\eps}{1-\eps}\mu_k(\Omega)+\frac{1+\frac1{\eps}}{1-\eps}C^2\mu_k(\Omega)^{2\theta}\right)^{\frac{n}2}\left|\Omega\right|.
	\end{equation*}
	We recall that  $\mu_k(\Omega)\le \lambda_k(\Omega)$ for all positive integer $k$ \cite[XIII.15, Proposition 4]{ReeSim78}. According to Weyl's law \cite[Theorem XIII.78]{ReeSim78},
	\begin{equation*}
		\lim_{k\to +\infty}\frac{\lambda_k(\Omega)^{\frac{n}2}|\Omega|}{k}=\frac{(2\pi)^n}{\omega_n}.
	\end{equation*}
	We therefore have the asymptotic upper bound
	\begin{equation}
	\label{eqIneqWeyl}
		\limsup_{k\to +\infty}\frac{\mu_k(\Omega)^{\frac{n}2}|\Omega|}{k}\le\frac{(2\pi)^n}{\omega_n}.
	\end{equation}
	This gives us
	\begin{equation*}
	\limsup_{k\to+\infty}\frac{\nu_0(\eps,\mu_k(\Omega))}{k}\le \left(\frac{1+\eps}{1-\eps}\right)^{\frac{n}2} \frac{(2\pi)^n}{\lambda_1(\Bn)^{\frac{n}2}\omega_n}=\left(\frac{1+\eps}{1-\eps}\right)^{\frac{n}2}\gamma(n).
	\end{equation*}
	On the other hand, Inequality \eqref{eqIneqN1} implies
	\begin{equation*}
	\limsup_{k\to+\infty}\frac{\nu_1(\eps,\mu_k(\Omega))}{k}\le \limsup_{k\to+\infty} \frac{\widetilde{C}'}{k\mu_k(\Omega)^{\theta}}\left(\frac{2}{\eps}\mu_k(\Omega)+\frac{2(C')^2}{\eps}\mu_k(\Omega)^{2\theta}\right)^{\frac{n}2},
	\end{equation*}
	and therefore, according to Inequality \eqref{eqIneqWeyl},
	\begin{equation*}
		\limsup_{k\to+\infty}\frac{\nu_1(\eps,\mu_k(\Omega))}{k}=0.
	\end{equation*}
	We obtain 
	\begin{equation*}
		\limsup_{k\to +\infty}\frac{\nu_k^N(\Omega)}{k} \le \left(\frac{1+\eps}{1-\eps}\right)^{\frac{n}2}\gamma(n).
	\end{equation*}
	Letting $\eps$ tend to $0$, we get finally
	\begin{equation*}
	\limsup_{k\to +\infty}\frac{\nu_k^N(\Omega)}{k} \le \gamma(n).
	\end{equation*}

\section{Proof of the auxiliary results}
\label{secIPP}

\subsection{Proof of Proposition \ref{propReg}}
 We use the following regularity result, which can be found for instance in \cite[2.1]{Grisvard1985}.
\begin{lem}
	\label{lemRegLp}
	If $\Omega$ is a bounded open set with a $C^{1,1}$ boundary and $f\in L^p\left(\Omega \right)$ with $p\in(1,\infty)$, there exists a unique $w\in W^{2,p}(\Omega)$ which solves
	\begin{equation*}
		\left\{\begin{array}{rcl}	
			-\Delta w+ w&=&f\mbox{ in } \Omega;\\
			 \frac{\partial w}{\partial \nu}&=&0 \mbox{ on } \partial \Omega.
			 \end{array}
			\right.
	\end{equation*}
\end{lem}
We now follow the method indicated in \cite[Remark 1.2.11]{Hen06}. Let us consider $u$, an eigenfunction of $-\Delta_{\Omega}^N$ associated with $\mu$. The function $u$ is in $H^1(\Omega)$, and is the unique weak solution of the boundary value problem 
\begin{equation*}
\left\{\begin{array}{rcl}	
-\Delta u+ u&=&f\mbox{ in } \Omega;\\
\frac{\partial u}{\partial \nu}&=&0 \mbox{ on } \partial \Omega;
\end{array}
\right.
\end{equation*}
with $f=(\mu+1)u$.
We know by Lemma \ref{lemRegLp} that this system has a solution $w\in H^2(\Omega)$. By uniqueness of the weak solution, $w=u$. We therefore obtain $u\in H^2(\Omega)$.

If $n\le 4$, the Sobolev embedding theorem tells us that for any $p\in [1,\infty[$, $u\in L^p(\Omega)$, and another application of Lemma \ref{lemRegLp} gives us $u\in W^{2,p}(\Omega)$. If $n>4$, we still obtain $u\in L^{p}(\Omega)$, and therefore $u\in W^{2,p}(\Omega)$, for all  $p\in [1,\infty[$, after a standard bootstrap argument, using repeatedly Lemma \ref{lemRegLp} and the Sobolev embedding theorem. Since $W^{2,p}(\Omega) \subset C^{1,1-\frac{n}p}\left(\overline \Omega \right)$ for all $p>n$, we have proved Proposition \ref{propReg}.
\subsection{Proof of Proposition \ref{propGreen}}
\label{subsecPropGreen}

We use the method of \cite[Appendix D]{BerMey82}. We consider an eigenfunction $u$ of $-\Delta_{\Omega}^N$ associated with $\mu$, and a nodal domain $D$ of $u$. Up to replacing $u$ by $-u$, we assume that $u$ is positive in $D$. Since $u\in C^{1}\left(\overline \Omega\right)$, there exists an open neighborhood $\mathcal O$ of $\overline \Omega$ in $\RR^n$ and a $C^{1}$ function $g: \mathcal O\to \RR$ such that $g=u$ in $\Omega$. We denote by $E$ the nodal domain of $g$ containing $D$. For $\alpha>0$ small enough, we write 
\begin{equation*}
	D_{\alpha}:=\{x\in D\,;\,u(x)>\alpha\}
\end{equation*}
and
\begin{equation*}
	E_{\alpha}:=\{x\in E\,;\,g(x)>\alpha\}.
\end{equation*}
Let us note that 
\begin{equation*}
	\partial E_{\alpha}\cap \mathcal O=\{x \in E\,;\, g(x)=\alpha\}.
\end{equation*}
We have the following decomposition of $\partial D_{\alpha}\subset \overline \Omega$ into disjoint subsets:
\begin{equation}
\label{eqDecompBound}
	\partial D_{\alpha}=\Sigma_{\alpha}\cup \Gamma_{\alpha}\cup \gamma_{\alpha},
\end{equation}
where
\begin{equation*}
	\Sigma_{\alpha}:=\partial E_{\alpha}\cap \Omega
\end{equation*}
is a closed set in $\Omega$,
\begin{equation*}
	\Gamma_{\alpha}:= E_{\alpha} \cap \partial \Omega
\end{equation*}
is an open set in $\partial \Omega$, and
\begin{equation*}
	\gamma_{\alpha}:= \partial E_{\alpha} \cap \partial \Omega
\end{equation*}
is a closed set in $\partial \Omega$.

We now assume that $\alpha$ is a regular value for the function $g$. Then $\partial E_{\alpha}$ is a $C^1$-regular surface, and for each $x \in \partial E_{\alpha}$, $\nabla g(x)$ is orthogonal to $\partial E_{\alpha}$ at $x$. Since $u$ satisfies a Neumann boundary condition on $\partial \Omega$, we have $\nu(x)\cdot \nabla g(x)=0$ for any $x\in \partial \Omega$, where $\nu(x)$ is the exterior normal unit vector to $\partial \Omega$ at $x$. This implies that the two $C^1$-submanifolds $\partial E_{\alpha}$ and $\partial \Omega$ intersect transversally, and therefore that $\gamma_{\alpha}$ is a $C^1$-submanifold of $\partial \Omega$, with dimension $n-2$. From this and the decomposition \eqref{eqDecompBound}, we conclude that $\partial D_{\alpha}$ is Lipschitz. We can therefore apply Green's formula to the function $u_{\alpha}:=u-\alpha$ in $D_{\alpha}$ (see \cite[Theorem 1.5.3.1]{Grisvard1985}), and we obtain
\begin{equation*}
\int_{D_{\alpha}} \left(-\Delta u_{\alpha}\right)u_{\alpha}\,dx=-\int_{\partial D_{\alpha}}u_{\alpha}\frac{\partial u_{\alpha}}{\partial \nu}\,d\sigma+\int_{D_{\alpha}}\left|\nabla u_{\alpha}\right|^2\,dx,
\end{equation*}  
and thus
\begin{equation*}
	\mu\int_{D_{\alpha}} u^2\,dx-\alpha\mu\int_{D_{\alpha}}u\,dx=-\int_{\Sigma_{\alpha}}u_{\alpha}\frac{\partial u_{\alpha}}{\partial \nu}\,d\sigma-\int_{\Gamma_{\alpha}}u_{\alpha}\frac{\partial u_{\alpha}}{\partial \nu}\,d\sigma+\int_{D_{\alpha}}\left|\nabla u\right|^2\,dx.
\end{equation*}
We have $u_{\alpha}=0$ on $\Sigma_{\alpha}$ and $\frac{\partial u_{\alpha}}{\partial \nu}=0$ on $\Gamma_{\alpha}$, and therefore
\begin{equation*}
	\mu\int_{D_{\alpha}} u^2\,dx-\alpha\mu\int_{D_{\alpha}}u\,dx=\int_{D_{\alpha}}\left|\nabla u\right|^2\,dx.
\end{equation*}
According to Sard's theorem, there exists a sequence $(\alpha_m)_{m\ge 1}$ of positive regular values for the function $g$, converging to $0$. For any $m$ large enough, we have
\begin{equation*}
\mu\int_{D_{\alpha_m}} u^2\,dx-\alpha_m\mu\int_{D_{\alpha_m}}u\,dx=\int_{D_{\alpha_m}}\left|\nabla u\right|^2\,dx.
\end{equation*}
Letting $m\to +\infty$, we get
\begin{equation*}
\mu\int_{D} u^2\,dx=\int_{D}\left|\nabla u\right|^2\,dx,
\end{equation*}
which concludes the proof of Proposition \ref{propGreen}.

\section{The Robin boundary condition}
\label{secRobin}

Let us begin this Section with a definition of the operator $-\Delta_{\Omega}^{R,h}$. We follow the method of \cite[3.1]{Rozenblum1994}, although this reference uses slightly stronger regularity assumption on the domain. We define the real bilinear form $q_h$ on the domain $H^1(\Omega)$ by
\begin{equation*}
	q_h(u,v)=\int_{\Omega}\nabla u\cdot\nabla v\,dx+\int_{\partial\Omega} uv\,d\sigma
\end{equation*}
for all $u$ and $v$ in $H^1(\Omega)$. The form $q_h$ is closed, symmetric, and non-negative.  We define the self-adjoint operator $-\Delta_{\Omega}^{R,h}$ as the Friedrichs extension of $q_h$ \cite[Theorem X.23]{ReeSim75}. The compact embedding $H^1(\Omega)\subset L^2(\Omega)$ ensures that $-\Delta_{\Omega}^{R,h}$ has compact resolvent. The spectrum of $-\Delta_{\Omega}^{R,h}$ therefore consists in a sequence of isolated non-negative eigenvalues of finite multiplicity tending to $+\infty$, which we denote by $(\mu_{k}(\Omega,h))_{k\ge 1}$ (with repetition according to the multiplicities). Let us point out that for all positive integer $k$, $\mu_k(\Omega,h)\le \lambda_k(\Omega)$. Indeed,
\begin{equation*}
	q_h(u,v)=\int_{\Omega}\nabla u\cdot\nabla v\,dx
\end{equation*}
for all $u$ and $v$ in $H^1_0(\Omega)$, so that the inequality follows from the minmax characterization of $\mu_k(\Omega,h)$ and $\lambda_k(\Omega)$ \cite[Theorem XIII.2]{ReeSim78}. Weyl's law for the sequence $(\lambda_k(\Omega))_{k\ge 1}$ then implies
\begin{equation}
\label{eqIneqWeylRob}
	\limsup_{k \to +\infty}\frac{\mu_k(\Omega,h)^{\frac{n}2}|\Omega|}{k}\le \frac{(2\pi)^n}{\omega_n}.
\end{equation}

Proposition \ref{propReg} can be generalized in the following way.
\begin{prop} \label{propRegRob} Let $\Omega \subset \RR^n$ be  an open, bounded, and connected set with a $C^{1,1}$ boundary, and $h$ be a Lipschitz function in $\overline \Omega$ with $h \ge 0$ on $\partial \Omega$. An eigenfunction $u$ of $-\Delta_{\Omega}^{R,h}$ belongs  to $C^{1,1^-}\left(\overline \Omega\right):=\bigcup_{\alpha \in (0,1)}C^{1,\alpha}\left(\overline \Omega\right)$. In particular, $u\in C^1\left(\overline \Omega\right)$. 
\end{prop}

To prove Proposition \ref{propRegRob}, we use the following regularity result, which is a special case of \cite[Theorem 2.4.2.7]{Grisvard1985}.

\begin{lem}
	\label{lemRegLpRob}
	Let $\Omega \subset \RR^n$ be  an open, bounded, and connected set with a $C^{1,1}$ boundary, and let $h$ be a Lipschitz function in $\overline \Omega$ with $h \ge 0$ on $\partial \Omega$, and $f\in L^p\left(\Omega \right)$ with $p\in(1,\infty)$. There exists a unique $w\in W^{2,p}(\Omega)$ which solves
	\begin{equation*}
	\left\{\begin{array}{rcl}	
	-\Delta w+ w&=&f\mbox{ in } \Omega;\\
	\frac{\partial w}{\partial \nu}+hw&=&0 \mbox{ on } \partial \Omega.
	\end{array}
	\right.
	\end{equation*}
\end{lem}
We then repeat the steps in the proof of Proposition \ref{propRegRob}. For the type of Robin boundary condition studied here, the Green identity given in Proposition \ref{propGreen} can be replaced by the following inequality.

\begin{prop}
	\label{propGreenRob} Let $\Omega \subset \RR^n$ be  an open, bounded, and connected set with a $C^{1,1}$ boundary, and let $h$ be a Lipschitz function in $\overline \Omega$ with $h \ge 0$ on $\partial \Omega$. If $u$ is an eigenfunction of $-\Delta_{\Omega}^{R,h}$ associated with the eigenvalue $\mu$, and if $D$ is a nodal domain of $u$, then
	\begin{equation}
	\label{eqIneqGreen}
		\int_{D}\left|\nabla u\right|^2\,dx \le	\mu\int_{D}u^2\,dx.
	\end{equation}
\end{prop}

\begin{proof}
	Using Proposition \ref{propRegRob} instead of Proposition \ref{propReg}, we essentially repeat the steps in the proof of Proposition \ref{propGreen}. However, two points have to be modified. First, since we do not in general have $\frac{\partial u}{\partial \nu}=0$ on $\partial \Omega$, the argument in Section \ref{subsecPropGreen} showing that $\partial E_{\alpha}$ and $\partial \Omega$ intersect transversally does not apply. However, if $\alpha>0$ is a regular value for the function $g:\partial \Omega\to \RR$, the tangential part of $\nabla g(x)$ is non-zero when $x\in \gamma_{\alpha}$. If $\alpha$ is also a regular value for $g$, we can proceed as in Section \ref{subsecPropGreen}. Applying Green's formula to $u_{\alpha}$, and using the Robin boundary condition, we obtain
	\begin{equation*}
		\mu\int_{D_{\alpha}} u^2\,dx-\alpha\mu\int_{D_{\alpha}}u\,dx=\int_{\Gamma_{\alpha}}hu u_{\alpha}\,d\sigma+\int_{D_{\alpha}}\left|\nabla u\right|^2\,dx.
	\end{equation*}
	Since $h\ge 0$ on $\partial \Omega$ and $u>0$ in $D_{\alpha}$, this implies
	\begin{equation}
	\label{eqIneqGreenApp}
	\mu\int_{D_{\alpha}} u^2\,dx\ge\int_{D_{\alpha}}\left|\nabla u\right|^2\,dx.
	\end{equation}
	Using Sard's theorem for $g$ and its restriction $g:\partial \Omega \to \RR$, we find a sequence $(\alpha_m)_{m\ge 1}$ of positive regular values for both functions, such that $\alpha_m\to 0$. We conclude by applying Inequality \eqref{eqIneqGreenApp} with $\alpha=\alpha_m$ and passing to the limit.
\end{proof}

The proof of Theorem \ref{thmPleijelRob} then follows the steps of Section \ref{secProof}, using Inequality \eqref{eqIneqWeylRob} instead of Inequality \eqref{eqIneqWeyl}, and Inequality \eqref{eqIneqGreen} instead of Equation \eqref{eqGreen}.

\paragraph{Acknowledgements} The author thanks Bernard Helffer for introducing him to this problem, for his suggestions and corrections, as well as advice and encouragement, and  Pierre B{\'e}rard for his careful reading of this work, for suggesting numerous improvements, and correcting the reflection argument of Section \ref{subSecRef}. The author also thanks Susanna Terracini for discussions of related results and possible extensions, Alessandro Iacopetti for his help with the regularity questions, and James Kennedy for discussions of the Robin case. This work was partially supported by the ANR (Agence Nationale de la Recherche), project OPTIFORM n$^\circ$ ANR-12-BS01-0007-02, and by the ERC, project COMPAT n$^\circ$ ERC-2013-ADG.

{\small
\bibliographystyle{plain}

\end{document}